\documentclass[11pt]{article}

\usepackage{amsmath}
\usepackage{amsthm}
\usepackage{amsfonts}
\usepackage{amssymb}
\usepackage{enumitem}
\setlist[enumerate]{label={\textnormal{(\roman*)}}}
\usepackage{array}
\usepackage{setspace}
\usepackage[labelformat=simple]{subcaption}

\usepackage{color}

\usepackage{hyperref}

\usepackage{url}
\usepackage{graphicx}
\usepackage[numbers,sort&compress]{natbib}
\usepackage[margin=1.5 in]{geometry}

\newtheorem{theorem}{Theorem}[section]

\newtheorem{lemma}[theorem]{Lemma}
\newtheorem{corollary}[theorem]{Corollary}
\newtheorem{observation}[theorem]{Observation}

\theoremstyle{definition}

\newtheorem{algorithm}[theorem]{Algorithm}

\newtheorem{question}[theorem]{Question}

\theoremstyle{remark}

\newcommand{\ceeil}[1]{\left \lceil #1 \right \rceil }
\newcommand{\alg}{Shortlex Assignment Algorithm}

\usepackage{tikz}
\usetikzlibrary{calc}
\usetikzlibrary{arrows}
\usetikzlibrary{decorations.markings}
\tikzset{->-/.style={decoration={
			markings,
			mark=at position #1 with {\arrow{>}}},postaction={decorate}}}

\tikzstyle{vertex}=[circle, draw, inner sep=0pt, minimum size=4pt,fill=black]
\newcommand{\vertex}{\node[vertex]}
\tikzstyle{hollowvertex}=[circle, draw, inner sep=0pt, minimum size=4pt, fill=white]
\newcommand{\hollowvertex}{\node[hollowvertex]}
\tikzstyle{phantomvertex}=[circle, draw, inner sep=0pt, minimum size=4pt,color=white]

\begin{document}
	\title{The Threshold Dimension and Irreducible Graphs}
\author{Lucas Mol\footnote{University of Winnipeg, 515 Portage Avenue, Winnipeg, MB R3B 2E9}, Matthew J. H. Murphy\footnote{University of Toronto, 27 King's College Circle, Toronto, ON M5S 1A1}, and Ortrud R.\ Oellermann$^*$\thanks{Supported by an NSERC Grant CANADA, Grant number RGPIN-2016-05237}\\
	{\small \href{mailto:l.mol@uwinnipeg.ca}{l.mol@uwinnipeg.ca}, \href{mailto:mattjames.murphy@mail.utoronto.ca}{mattjames.murphy@mail.utoronto.ca},}\\ 
{	\small\href{mailto:o.oellermann@uwinnipeg.ca}{o.oellermann@uwinnipeg.ca}}}
\date{}

\maketitle
\begin{abstract}

Let $G$ be a graph, and let $u$, $v$, and $w$ be vertices of $G$. If the distance between $u$ and $w$ does not equal the distance between $v$ and $w$, then $w$ is said to \emph{resolve} $u$ and $v$.	The \emph{metric dimension} of $G$, denoted $\beta(G)$, is the cardinality of a smallest set $W$ of vertices such that every pair of vertices of $G$ is resolved by some vertex of $W$.
The \emph{threshold dimension} of $G$, denoted $\tau(G)$,  is the minimum metric dimension among all graphs $H$ having $G$ as a spanning subgraph. In other words, the threshold dimension of $G$ is the minimum metric dimension among all graphs obtained from $G$ by adding edges.  If $\beta(G) = \tau(G)$, then $G$ is said to be \emph{irreducible}.

	We give two upper bounds for the threshold dimension of a graph, the first in terms of the diameter, and the second  in terms of the chromatic number.  As a consequence, we show that every planar graph of order $n$ has threshold dimension $ O (\log_2 n)$. We show that several infinite families of graphs, known to have metric dimension $3$, are in fact irreducible.  Finally, we show that for any integers $n$ and $b$ with $1 \leq b < n$, there is an irreducible graph of order $n$ and metric dimension $b$.
	
\end{abstract}

\section{Introduction} \label{introduction}

Slater \cite{Slater1975}, being motivated by the problem of uniquely determining the location of an intruder in a network, first introduced the notion of `resolvability' in graphs.  For vertices $x$ and $y$ of a graph $G$, let $d_G(x,y)$ denote the distance between $x$ and $y$ in $G$.  We write $d(x,y)$ in place of $d_G(x,y)$ if $G$ is clear from context.  A vertex $w$ is said to \emph{resolve} a pair $u,v$ of vertices in $G$ if $d(u,w)\neq d(v,w)$.
A set $W \subseteq V(G)$ of vertices \emph{resolves} the graph $G$, and we say that $W$ is a \emph{resolving set} for $G$, if every pair of vertices of $G$ is resolved by some vertex of $W$. A smallest  resolving set of $G$ is called a \emph{basis} of $G$, and its cardinality is called the \emph{metric dimension} of $G$, denoted $\beta(G)$. Since being introduced by Slater \cite{Slater1975}, and independently by Harary and Melter \cite{HararyMelter1976},  the metric dimension has been studied extensively.
See the work of C\'aceres et al.~\cite{Caceresetal2007} for an extensive list of publications related to the theoretical aspects of the metric dimension, and the work of Belmonte et al.~\cite{Belmonteetal2015} for an extensive list of publications related to the computational aspects of the metric dimension.  Henceforth, when we say dimension in this paper, unless qualified, we are referring to the metric dimension.

The question of how the metric dimension of a graph relates to that of its subgraphs has been studied, for example, by Chartrand et al.~\cite{Chartrandetal2000} and Khuller et al.~\cite{Khulleretal1996}. In this article, we focus on the metric dimension of those graphs that have a given graph $G$ as a \emph{spanning} subgraph. Suppose that distance detecting devices can be installed at nodes (vertices) of a network $G$ that indicate the distance to an intruder in the network. If $W$ is a resolving set for $G$, and if a detecting device is installed at each node of $W$, then these devices can uniquely determine the location of an intruder in the network.  It is natural to ask whether the number of detecting devices that are needed can be reduced if additional edges are added to the existing network.  The \emph{threshold dimension} of a graph $G$, denoted $\tau(G)$, is defined as $\min\{\beta(H)\colon\ H ~\mbox{contains}~ G ~\mbox{as a spanning subgraph}\}$.  A graph $H$ having $G$ as a spanning subgraph and such that $\beta(H)=\tau(G)$ is called a \emph{threshold graph} of $G$. A graph $G$ is called \emph{irreducible} if $\tau(G)=\beta(G)$; otherwise, it is called \emph{reducible}.

The threshold dimension of a graph was introduced in a recent article by the current authors~\cite{us}, in which the following statements were proven:
\begin{itemize}
\item There is a geometric interpretation of the threshold dimension of a graph, in terms of a minimum number of strong products of paths (each of sufficiently large order) that admits a certain type of embedding of the graph.  
\item Every tree $T$ with $\beta(T) \geq 3$ is reducible.
\item Every tree with dimension $3$ or $4$ has threshold dimension $2$.
\item There exist trees with arbitrarily large metric dimension having threshold dimension $2$.
\end{itemize}


In this article, we continue the study of the threshold dimension of a graph.  Section~\ref{preliminaries} is devoted to some preliminaries. In Section~\ref{generalbounds}, we present two upper bounds on the threshold dimension of a graph $G$; the first in terms of the diameter of $G$, and the second in terms of the chromatic number of $G$. The latter bound is shown to be sharp. In Section~\ref{irrgraphs}, we focus on irreducible graphs.  We show that the highly symmetric graphs of metric dimension $3$ studied by Javaid et al.~\cite{Javaidetal2008} are irreducible, and that for every $n\geq 4$ and $b\in\{3,\dots, n-1\}$, there is an irreducible graph of order $n$ and dimension $b$.

\section{Preliminaries} \label{preliminaries}

Let $G$ be a graph. We let $V(G)$ denote the vertex set of $G$, and $E(G)$ denote the edge set of $G$.  The diameter of $G$ is denoted $\mbox{diam}(G)$.  A shortest path between two vertices $u,v\in V(G)$ is called a \emph{diametral path} of $G$ if it has length $\mbox{diam}(G)$.  The minimum degree among all vertices of $G$ is denoted $\delta(G)$, and the maximum degree among all vertices of $G$ is denoted $\Delta(G)$.  The chromatic number of $G$ is denoted $\chi(G)$.  We adopt the convention that $\beta(K_1)=0$.

The complement of $G$ is denoted by $\overline{G}$.  For any set $S\subseteq E\left(\overline{G}\right)$, the graph obtained from $G$ by adding the edges of $S$ is denoted by $G+S$.  For disjoint graphs $G$ and $H$, the \emph{join} of $G$ and $H$, denoted $G \vee H$, is the graph on vertex set $V(G)\cup V(H)$ and edge set $E(G)\cup E(H)\cup\{uv\colon\ u\in V(G) \mbox{ and } v \in V(H)\}$.  For a graph $G$ and a positive integer $k$, let $G^k$ denote the $k$th power of $G$, that is, the graph with vertex set $V(G)$ and edge set $\{uv\colon\ d_G(u,v)\leq k\}$.

Let $v$ be a vertex of $G$.  The $k$-\emph{neighbourhood} of $v$ in $G$, denoted $N_k^G(v)$, is the set of vertices in $G$ whose distance from $v$ is exactly $k$.  We usually use $N^G(v)$ instead of $N_1^G(v)$.  Let $W\subseteq V(G)$. Then the $W$-\emph{neighbourhood} of $v$ in $G$, denoted $N^G_W(v)$, is defined as $N_W^G(v)=N^G(v) \cap W$. A vertex $v$ is said to be $W$-\emph{universal} in $G$ if $N_W^G(v)=W$; i.e., if $v$ is adjacent to all vertices in $W$.  Whenever the graph $G$ is clear from context, we omit the superscript in $N_k^G(v)$ and $N_W^G(v)$.

In the sequel, the key idea used in establishing an upper bound for the threshold dimension of a graph $G$ is to find a set $W$ of vertices for which it is possible to add edges to $G$ in such a manner that every two vertices in $G-W$ have distinct $W$-neighborhoods in the resulting graph. For a set $S$, we use $\mathcal{P}(S)$ to denote the power set of $S$.
Let $G$ be a graph, and let $P,W\subseteq V(G)$ satisfy the following conditions:
\begin{enumerate}
\item $W\cap P=\emptyset$;
\item $|P|\leq 2^{|W|}$; and
\item for every pair of vetices $u,v\in P$ with nonempty $W$-neighbourhoods, we have $N_W(u)\neq N_W(v)$.
\end{enumerate}
Then the following algorithm outputs a set of edges $\mathcal{E}$ such that all vertices in $P$ have distinct $W$-neighbourhoods in $G+\mathcal{E}$.  In particular, this means that $W$ resolves $P$ in $G+\mathcal{E}$.

\begin{algorithm}[The Shortlex Assignment Algorithm] \label{AA}
 Input a graph $G$, along with sets $P,W \subseteq V(G)$ satisfying conditions (i)-(iii) above. Let $W=\{x_1,\dots,x_k\}$ and $P=\{u_1,\dots,u_r\}\cup\{v_1,\dots,v_s\}$, where $N_W^G(u_i)\neq \emptyset$ for all $i\in\{1,\dots,r\}$ and $N_W^G(v_j)=\emptyset$ for all $j\in\{1,\dots,s\}$.
\begin{enumerate}
\item \textbf{Collect subsets:} Let $\mathcal{N}=\{N_W(u_i)\colon\ 1 \leq i \leq r\}$. Let $\mathcal{S}=\mathcal{P}(W)-\mathcal{N}$.
\item \textbf{Sort subsets:} Sort $\mathcal{S}$ using the shortlex ordering, i.e., sort first by cardinality, with the smallest subsets appearing first, and then lexicographically within each cardinality, with $x_1 < ... < x_k$. Let $\mathcal{S}=\{S_1,\dots,S_{|\mathcal{S}|}\}$, where $|S_1|<\ldots < |S_{|\mathcal{S}|}|$. (By condition (ii), we have $|\mathcal{S}|=2^{|W|}-r \geq |P|-r=s$.)
\item \textbf{Output edges:} For all $j\in \{1,\dots,s\}$, let $E_j=\{v_jx\colon\ x \in S_{j}\}$. Let $\mathcal{E}=\bigcup_{j=1}^sE_j$.  Output $\mathcal{E}$.
\end{enumerate}
\end{algorithm}

In the graph $G+\mathcal{E}$, the $W$-neighbourhood of $v_j$ is exactly $S_j$.  In other words, the algorithm \emph{assigns} the $W$-neighbourhood $S_j$ to the vertex $v_j$.
The shortlex ordering guarantees that some vertex in $P$ has empty $W$-neighbourhood in $G+\mathcal{E}$ (unless $s=0$), and that no vertex in $P$ is $W$-universal in $G+\mathcal{E}$ (unless some $u_i$ is $W$-universal or $|P|=2^{|W|}$).  We will also make use of the \emph{Reverse Shortlex Assignment Algorithm}, which is the same as the Shortlex Assignment Algorithm, except at step 2, the reverse shortlex ordering is used.  This guarantees that some vertex in $P$ is assigned the entire set $W$ (unless $s=0$), i.e., that some vertex in $P$ is $W$-universal in $G+\mathcal{E}$.

\smallskip

We now state two elementary results which will be useful in several parts of the paper.
%
%
%
The first is a generalization of the fact that no graph of metric dimension $2$ has $K_5$ as a subgraph~\cite[Theorem 3.2]{Khulleretal1996}. Khuller et al.~\cite{Khulleretal1996} noted that the result could be generalized, and the proof of the following lemma is indeed straightforward.


\begin{lemma}\label{kn}
	Let $G$ be a graph with $K_n$ as a subgraph. Then $\beta(G)\geq \ceeil{\log_2 n}.$
\end{lemma}



We also use a tight lower bound on the metric dimension of a graph of order $n$ and diameter $2$, proven by Khuller et al.~\cite{Khulleretal1996}, and independently by Chartrand et al.~\cite{Chartrandetal2000}.  We note that a tight bound on the metric dimension of a graph of any given order $n$ and diameter $d$ was later proven by Hernando et al.~\cite{Hernandoetal2010}, but we only need the special case $d=2$.  Define $g: (1, \infty) \rightarrow \mathbb{N}$ as follows: $g(x)$ is the smallest integer $d$ such that $2^d+d\geq x$, i.e., we have $g(x)=d$ if and only if $x \in (2^{d-1}+d-1, 2^d+d]$.

\begin{lemma}\label{weakbound}
Let $G$ be a graph of order $n$ and diameter $2$. Then $\beta(G)\geq g(n)$.
\end{lemma}

\section{Bounds on the threshold dimension of a graph} \label{generalbounds}

In this section, we prove upper bounds on the threshold dimension of a graph.  We begin with a general result from which a bound in terms of diameter follows in a straightforward manner.  The proof relies on a process similar to that of the \alg.

\begin{theorem} \label{diammeth}
Let $G$ be a graph of order $n$. Suppose that there exists a set $W \subseteq V(G)$ and an integer $\ell\in\{0,1,\ldots,|W|-1\}$ such that:
\begin{enumerate}
\item $2^{|W|-\ell}\geq n-|W|$; and
\item for every $x \in V(G)-W$, we have $N_W(x)\leq \ell$.
\end{enumerate}
Then $\tau(G) \leq |W|$.
\end{theorem}

\begin{proof}
Let $P= V(G)-W$. We show that we can add edges to $G$ so that every vertex $x \in P$ has a unique $W$-neighbourhood.
	
Let $P=\{v_1, ..., v_k\}$, where $k=n-|W|$. We assign to each $v_i$ a distinct subset $S_i$ of $W$ containing $N_W^G(v_i)$ as follows.  We begin by assigning $v_1$ the subset $S_1=N_W^G(v_1)$. Now let $i\geq 2$ and suppose that $v_1, \ldots, v_{i-1}$ have been assigned distinct subsets $S_1,\dots, S_{i-1}$ of $W$ that contain $N_W^G(v_1),\dots,N_W^G(v_{i-1})$, respectively. Since we have $2^{|W|-|N_W^G(v_i)|} \ge 2^{|W|-\ell} \geq n-|W|=k \geq i$, there is a subset $S_i$ of $W$ containing $N_W^G(v_i)$ that is distinct from $S_1,\dots,S_{i-1}$.  Assign $v_i$ the subset $S_i$.

Now let $H=G+\mathcal{E}$, where $\mathcal{E}=\bigcup_{i=1}^k \{v_is\colon\ s\in S_i\backslash N_W^G(v_i)\}$.  For all $i\in\{1,\dots,k\}$, we have $N_W^H(v_i)=S_i$, and since the $S_i$ are all distinct, we conclude that $W$ is a resolving set for $H$.  Therefore, $\tau(G)\leq |W|$.
\end{proof}

As a consequence of Theorem~\ref{diammeth}, we see that if a graph $G$ has sufficiently large diameter (relative to its order), then the threshold dimension of $G$ is bounded above by its diameter.  In fact, if the diameter of $G$ is large enough, then for any diametral path $D$ in $G$, there is a graph $H$ containing $G$ as a spanning subgraph in which $V(D)$ is a resolving set.

\begin{corollary}
Let $G$ be a graph of order $n$ and diameter $d$. If $2^{d-3}\geq n-d$, then $\tau(G)\leq d$.
\end{corollary}

\begin{proof}
Suppose that $2^{d-3}\geq n-d$.  Let $D$ be a diametral path of $G$, and let $W=V(D)$.  Note that every vertex in $V(G)-W$ is adjacent to at most three vertices in $W$, since $D$ is a diametral path.  The result now follows immediately from Theorem~\ref{diammeth}.
\end{proof}


\medskip

We now work towards a bound on the threshold dimension for any graph of order $n$ and chromatic number $k$.  By the following straightforward observation, it suffices to bound the threshold dimension of all complete $k$-partite graphs of order $n$.

\begin{observation} \label{chromobs1}
Let $H$ be a graph that contains $G$ as a spanning subgraph.  Then $\tau(G)\leq \tau(H)$.
\end{observation}

The next result gives the exact value of the threshold dimension of every complete multipartite graph.  Before we proceed, we make some preliminary observations and introduce some notation.

Let $K=K_{x_1,..., x_k}$ be a complete $k$-partite graph. Let $X_1, ..., X_k$ be the partite sets of $K$, where $|X_i|=x_i$ for $1 \leq i \leq k$. Let $H$ be a threshold graph of $K$, and let $W$ be a metric basis for $H$. Let $u$ and $v$ be two vertices of $V(H)-W$. If $u$ and $v$ belong to two distinct partite sets of $K$, say $u \in X_i$ and $v \in X_j$, then either $W$ contains a vertex of $X_i$ which is not adjacent in $H$ to $u$, or $W$ contains a vertex of $X_j$ which is not adjacent in $H$ to $v$. On the other hand, if $u$ and $v$ belong to the same partite set $X_i$, then $W$ must contain some vertex in $X_i$ which resolves $u$ and $v$.

Define the function $f:[1, \infty) \rightarrow \mathbb{N}$ as follows: $f(x)$ is the smallest integer $d$ such that $2^d+d>x$, i.e., we have $f(x)=d$ if and only if $x \in  [2^{d-1}+d-1, 2^d+d)$.  For a given integer $d \ge 1$, define $\ell_d = 2^{d-1}+d-1$; this is the smallest number that $f$ maps to $d$, and these numbers play an important role in the proofs that follow.  Note that we have $\ell_{d}-\ell_{d-1}=2^{d-2}+1$.

\begin{lemma} \label{chrom5}
Let $K=K_{x_1, ..., x_k}$ be a complete $k$-partite graph, and let $S_K=\sum_{i=1}^kf(x_i)$. Then
\[
\tau(K) = T_K:= \begin{cases}
S_K, &\text{if $x_i \ne \ell_{f(x_i)}$ for every $1 \le i \le k$;}\\
S_K-1, &\text{otherwise}.
\end{cases}
\]
\end{lemma}

\begin{proof}
	Let $K$ and $c$ be as above, and let $X_1, X_2, \ldots, X_k$ be the partite sets of $K$, where $|X_i|=x_i$ for $1 \le i \le k$.  We first show that $\tau(K)\leq T_K$.
	
	First suppose that $x_i\neq\ell_{f(x_i)}$ for every $i\in\{1,\ldots,k\}$.  We construct a graph $H$ that contains $G$ as a spanning subgraph and has a resolving set of cardinality $S_K$.  For all $i\in\{1,\dots,k\}$, let $W_i$ be a set of $f(x_i)$ vertices from $X_i$, and let $P_i=X_i-W_i$.  Since $X_i$ is an independent set in $G$, the $W_i$-neighbourhood of every vertex in $P_i$ is empty.  By the definition of $f$, we have $x_i<2^{f(x_i)}+f(x_i),$ hence
\[
|P_i|=|X_i|-|W_i|=x_i-f(x_i)<2^{f(x_i)}=2^{|W_i|}.
\]
Thus, for every $i$, we may apply the Shortlex Assignment Algorithm with inputs $G$, $W_i$, and $P_i$, and no vertex of $P_i$ is assigned the entire set $W_i$.  Let $\mathcal{E}_i$ be the set of edges output by the algorithm.  Let $\mathcal{E}=\bigcup_{i=1}^k\mathcal{E}_i$, and define $H=G+\mathcal{E}$.  We claim that $W=\bigcup_{i=1}^k W_i$ is a resolving set for $H$.  Let $u$ and $v$ in $V(K)-W$.  If $u$ and $v$ belong to the same set $X_i$, then $N_{W_i}(u)\neq N_{W_i}(v)$, and thus $u$ and $v$ are resolved by some vertex in $W_i$.  Otherwise, we may assume that $u\in X_i$ and $v\in X_j$ for some $i<j$ (switching the labels of $u$ and $v$ if necessary).  Then $u$ is $W_j$-universal, while $v$ is not.  So $u$ and $v$ are resolved by some vertex from $X_j$.

	Now suppose that $x_i=\ell_{f(x_i)}$ for some $i\in\{1,\ldots,k\}$.  Without loss of generality, say $x_1=\ell_{f(x_1)}$.  We construct a graph $H$ that contains $G$ as a spanning subgraph and has a resolving set of cardinality $S_K-1.$  Let $W_1$ be a set of $f(x_1)-1$ vertices from $X_1$, and for all $i\geq 2$, let $W_i$ be a set of $f(x_i)$ vertices from $X_i$.  For all $i\in\{1,\dots,k\}$, let $P_i=X_i-W_i$.  Since $x_1=\ell_{f(x_1)}=2^{f(x_1)-1}+f(x_1)-1$, we have
\[
|P_1|=|X_1|-|W_1|=x_1-f(x_1)+1=2^{f(x_1)-1}=2^{|W_1|}.
\]
As above, we have $|P_i|<2^{|W_i|}$ for all $i\geq 2$.  Thus, for every $i$, we may apply the Shortlex Assignment Algorithm with inputs $G$, $W_i$, and $P_i$.  Since $|P_1|=2^{|W_1|}$, some vertex of $P_1$ is assigned the entire set $W_1$, but for every $i\geq 2$, no vertex of $P_i$ is assigned the entire set $W_i$.  Let $\mathcal{E}_i$ be the set of edges output by the algorithm.  Let $\mathcal{E}=\bigcup_{i=1}^k\mathcal{E}_i$, and define $H=G+\mathcal{E}$.  We claim that $W=\bigcup_{i=1}^k W_i$ is a resolving set for $H$.  The proof is the same as in the previous case.  (We insisted that $i<j$ in the previous case so that $v\not\in X_1$, guaranteeing that $v$ is not $W_j$-universal.)
	
%
	
We now prove that $\tau(K)\geq T_K$. Let $H$ be a graph containing $K$ as a spanning subgraph. Let $W\subseteq V(H)$ be a resolving set for $H$. We show that $|W| \geq T_K$. From the remark prior to Lemma~\ref{chrom5}, no vertex from $V(H)-X_i$ resolves any pair of vertices from $X_i$, for every $1 \leq i \leq k$.  So $W_i:=W\cap X_i$ must resolve $X_i$. Further, since $K$ (and hence $H$) has diameter $2$, all vertices in $X_i-W_i$ must have distinct $W_i$-neighbourhoods.  It follows that we must have $2^{|W_i|}\geq |X_i|-|W_i|$, or equivalently $2^{|W_i|}+|W_i|\geq |X_i|$.
	
First of all, if $x_i\neq \ell_{f(x_i)}$, then we have $x_i>2^{f(x_i)-1}+f(x_i)-1$.  It follows that $|W_i|\geq f(x_i)$.  Otherwise, if $x_i=\ell_{f(x_i)}=2^{f(x_i)-1}+f(x_i)-1$, then we must have $|W_i|\geq f(x_i)-1$.

Now suppose that there exist distinct integers $i$ and $j$ such that $|W_i|=f(x_i)-1$ and $|W_j|=f(x_j)-1$.  Then $x_i=\ell_{f(x_i)}=2^{|W_i|}+|W_i|$ and $x_j=\ell_{f(x_j)}=2^{|W_j|}+|W_j|$.  It follows that some vertex $v_i$ of $X_i-W_i$ is $W_i$-universal, and some vertex $v_j$ of $X_j-W_j$ is $W_j$-universal.  But then $v_i$ and $v_j$ are both $W$-universal, and hence $W$ does not resolve $H$, a contradiction.  This completes the proof that $\tau(K)\geq T_K$.
%
\end{proof}

We now establish a sharp upper bound for the threshold dimension of graphs of order $n$ and chromatic number $k$.

\begin{theorem}\label{chrombound}
	Let $G$ be a graph of order $n$ with $\chi(G)=k$. Then \[\tau(G) \leq k(f(n/k)+1)-1.\] Moreover, this bound is sharp for all $k$.
\end{theorem}

\begin{proof}
	Let $X_1, ..., X_k$ be a partition of $V(G)$ into $k$ nonempty independent sets. Let $|X_i|=x_i$ for $1 \leq i \leq k$, and assume without loss of generality that $x_1 \leq ... \leq x_k$. By Observation \ref{chromobs1}, it is sufficient to show that $\tau(K_{x_1, ..., x_k}) \leq k(f(n/k)+1)-1$. Let $d=f(n/k)$, and let $f(x_i)=d_i$ for $1 \leq i \leq k$. By Lemma \ref{chrom5}, we have $\tau(K_{x_1, ..., x_k}) \leq \sum_{i=1}^{k}d_i$. Observe that $d_1 \leq d$,  otherwise  $n = \sum_{i=1}^kx_i \geq k(2^{d}+d)$, which implies that $f(n/k)\geq d+1$, a contradiction. Further, if we have $d_i \leq d+1$ for all $2\leq i\leq k$, then the statement holds. So suppose that $d_k \ge d+2$.  We demonstrate the existence of a set $\{x^*_1, ..., x^*_k\}$ of positive integers such that
\begin{enumerate}[label=(\roman*)]
	\item $\sum_{i=1}^k x^*_i=n$;
	\item $\tau(K_{x^*_1, ..., x^*_k}) \geq \tau(K_{x_1, ..., x_k})$; and
	\item for all $1 \leq i \leq k$, we have $f(x^*_i) \leq d+1$.
\end{enumerate}
We claim that the theorem statement follows from this fact.  Suppose that such a set exists, and reorder if necessary so that $x_1^*\leq \cdots \leq x^*_k$.  Then by the above argument, we have $f(x^*_1)\leq d$, and together with (ii) and (iii), this gives the theorem statement.

Define $x'_k=\ell_{d_k-1}+1$ and $x'_1=x_1+x_k-x'_k$, and for all $2 \leq i \leq k-1$, define $x'_i=x_i$.  Note that $\sum_{i=1}^{k}x'_i=n$ and $f(x'_k)=f(x_k)-1$.  Since $d_k\geq d+2$, we also have
\begin{align}\label{eqn1}
x_k-x'_k \geq \ell_{d_k}-\ell_{d_k-1}-1=2^{d_k-2} \ge 2^d.
\end{align}
We now show that $\tau(K_{x'_1,\dots, x'_k})\geq \tau(K_{x_1,\dots,x_k})$.  Since $x_1 \in [2^{d_1-1}+d_1-1, 2^{d_1}+d_1)$ (and $d_1\leq d$), we may consider the following three cases.
	
	\noindent \textbf{Case 1:} $f(x'_1)=f(x_1)+1$ and $x'_1 = \ell_{d_1+1}$.
	
\noindent In this case, we have $2^{d_1} + d_1=x'_1= x_1+x_k-x'_k \ge x_1+2^d \ge 2^{d_1-1}+d_1-1+2^d$. So $2^{d_1-1} + 1 \ge 2^d$. Since $d \ge d_1$, this is only possible if $d=d_1=1$ and $x_k-x'_k =2$. Thus, by (\ref{eqn1}), we have $d_k=3$ and $x_k=\ell_3=6$.  Hence $x'_1=\ell_{f(x'_1)}$ and $x_k=\ell_{f(x_k)}$.  By Lemma~\ref{chrom5}, it follows that
\[
\tau(K_{x'_1,\dots, x'_k})=\sum_{i=1}^kf(x'_i)-1=\sum_{i=1}^kf(x_i)-1=\tau(K_{x_1,\dots,x_k}).
\]
	
	\smallskip
	
	\noindent \textbf{Case 2:} $f(x'_1)=f(x_1)+1$ and  $x'_1 \neq \ell_{d_1+1}$.
	
	\noindent In this case $\sum_{i=1}^kf(x'_i)=\sum_{i=1}^kf(x_i)$.  Further,  $x'_1\neq \ell_{f(x'_1)}$ and $x'_k\neq\ell_{f(x'_k)}$. Hence if $x'_i=\ell_{f(x'_i)}$ for some $i\in\{1,\dots, k\}$, then in fact  $i\in \{2,\dots, k-1\}$, and $x_i=\ell_{f(x_i)}$ as well.  It follows that $\tau(K_{x'_1,\dots, x'_k})=\tau(K_{x_1,\dots,x_k}).$
	
	\smallskip
	
	\noindent \textbf{Case 3:} $f(x'_1) \geq f(x_1)+2$.

	\noindent In this case,
	\[
	\tau(K_{x'_1,\dots,x'_k})\geq \sum_{i=1}^k f(x'_i)-1\geq \sum_{i=1}^k f(x_i)\geq \tau(K_{x_1,\dots,x_k}).
	\]

\smallskip

\noindent
This completes the proof that $\tau(K_{x'_1,\dots, x'_k})\geq \tau(K_{x_1,\dots,x_k})$.

Finally, note that $x'_1 < x_k$.  It follows that $\max\{x'_i\colon\ 1 \le i \le k\} \leq $ $\max\{x_i\colon 1 \le i \le k\}$, with equality if and only if $x_{k-1} = x_k$. Hence, we may repeatedly apply the entire process described above, and we will eventually reach a set $\{x^*_1,\dots,x^*_k\}$ satisfying conditions (i)-(iii).
	
	
	We now prove that the bound is sharp. For every $k\geq 2$, we show that there is an infinite family of complete $k$-partite graphs whose threshold dimension meets the given upper bound.  Fix $k\geq 2$.  Let $d_0$ be the smallest positive integer such that $2^{d_0-2}+1>k$. (Equivalently, $d_0$ is the smallest positive integer such that $\ell_{d_0+1}-\ell_{d_0}>k$.) For any $d\geq d_0$, let $x_{1}= \ell_{d}+1$, and let $x_{i} =\ell_{d+1}+1$ for all $i\in\{2,\dots,k\}$.  Let $n= \sum_{i=1}^{k}x_{i}$. 
Using the fact that $\ell_{d+1}-\ell_{d}>k$, it is easy to verify that $\ell_{d}<n/k<\ell_{d+1}$. Hence  $f(n/k)=d$. By Lemma \ref{chrom5}, we have 
\[
\tau(K_{x_{1},\dots,x_{k}})=(k-1)(d+1)+d=k(d+1)-1=k(f(n/k)+1)-1,
\]
as desired.
\end{proof}

By the Four Colour Theorem and Theorem~\ref{chrombound}, we obtain the following upper bound on the threshold dimension of every planar graph.

\begin{corollary}
Let $G$ be a planar graph of order $n$. Then
\[
\tau(G) \leq 4(f(n/4)+1)-1\leq 4\ceeil{\log_2(n)}-5.
\]
\end{corollary}


\section{Irreducible Graphs} \label{irrgraphs}

In this section, we focus on finding irreducible graphs. In our previous paper~\cite{us}, we mentioned that every graph of order $n$ and metric dimension $1$, $2$, or $n-1$ is irreducible.  But in general, it seems that irreducible graphs are more difficult to find than reducible graphs.  In Subsection~\ref{familiesirrpet}, we present two infinite families of graphs which are known to have metric dimension $3$, and we show that these graphs are irreducible.  In Subsection~\ref{prop}, we construct an irreducible graph of every order $n$ and dimension $b$, where $1\leq b<n$.

\subsection{Some irreducible graphs of dimension 3} \label{familiesirrpet}

We begin by showing that every graph of dimension $3$ and minimum degree at least $4$ is irreducible.  This is actually a straightforward corollary of the following result proven by Hernando et al.~\cite{Hernandoetal2010}, and reproven in our earlier work on the threshold dimension~\cite{us}.

\begin{lemma}\label{2dim}
Let $G$ be a graph with dimension $b$. If $\{w_1, ..., w_b\}$ is a basis for $G$, then for each $1 \leq i \leq b$  and for each $1 \leq k \leq \mathrm{diam}(G)$, we have $|N_k(w_i)| \leq (2k+1)^{b-1}$.
\end{lemma}

In particular, if $G$ has metric dimension $2$, then no vertex of degree at least $4$ belongs to a metric basis of $G$.  This immediately gives the following.

\begin{corollary}\label{mindeg}
Let $G$ be a graph.  If $\delta(G)\geq 4$, then $\beta(G)\geq \tau(G)\geq 3$.
\end{corollary}

For example, Corollary~\ref{mindeg} implies that $C_n^2$, the square of the cycle of order $n$, has threshold dimension at least $3$ for $n\geq 5$.  It was shown by Javaid et al.~\cite{Javaidetal2008} that $\beta(C_n^2)=3$ for $n\geq 6$ and $n\not\equiv 1\pmod{4}$, so we conclude that $\tau\left(C_n^2\right)=3$ for these values of $n$.\footnote{The graph $C_n^2$ is an example of a \emph{Harary graph}, and is denoted $H_{4,n}$ by Javaid et al.~\cite{Javaidetal2008}.  The graph $C_{2n}^2$ is also called an \emph{anti-prism graph}, and is denoted $A_n$ by Javaid et al.~\cite{Javaidetal2008}.}

Next, we establish that the threshold dimension of an infinite family of generalized Petersen graphs is $3$. This requires more work, as these graphs are $3$-regular, meaning that we cannot apply Corollary~\ref{mindeg}.  Let $P(n,k)$ denote the generalized Petersen graph with parameters $n$ and $k$, that is, the graph with vertex set $\{u_1, ..., u_n, v_1, ..., v_n\}$, and edge set $\{v_iv_{i+1}, u_iv_i, u_iu_{i+k}\colon\ 1 \leq i \leq k\}$, with indices taken modulo $n$. We call $\{u_1, ..., u_n\}$ the \emph{inner ring} of $P(n,k)$, and we call $\{v_1, ..., v_k\}$ the \emph{outer ring} of $P(n,k)$.  The graphs $P(6,2)$ and $P(8,2)$ are illustrated in Figure~\ref{Petersen}.

Javaid et al.~\cite{Javaidetal2008} showed that $\beta(P(n,2))=3$ for all $n\geq 5$.  Sudhakara et al.~\cite{Sudhakaraetal2009} demonstrated that no graph of dimension $2$ contains the Petersen graph as a subgraph, from which it follows that $P(5,2)$ is irreducible.  We now prove that $P(n,2)$ is in fact irreducible for all $n\geq 5$.

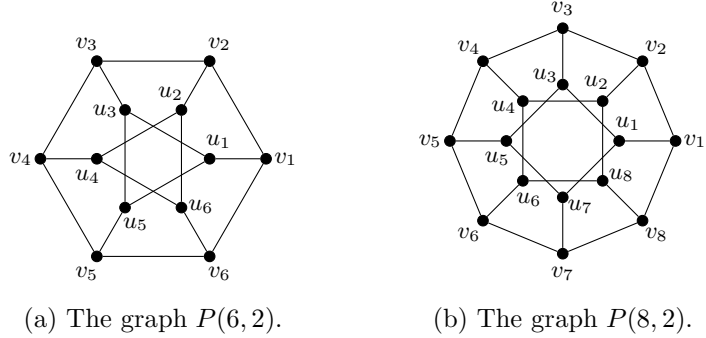
\begin{figure} 
	\centering
	\begin{subfigure}[t]{0.3\textwidth}
		\centering
		\begin{tikzpicture} [scale=0.75]
		\pgfmathtruncatemacro{\n}{6}
		\pgfmathtruncatemacro{\m}{\n-1}
		\pgfmathtruncatemacro{\a}{360/\n}

		\foreach \x in {0, ..., \m}
		{
			\pgfmathtruncatemacro{\p}{\x*\a}
			\vertex (x\x) at (\p:1) {};
			\vertex (y\x) at (\p:2) {};
		}
		
		\foreach \x in {0, 1, ..., \m}
		{
			
			\pgfmathtruncatemacro{\y}{mod(\x+1,\n)}
			\pgfmathtruncatemacro{\l}{mod(\x+2,\n)}
			
			\path
			(y\x) edge (y\y)
			(x\x) edge (y\x)
			(x\x) edge (x\l);
		}
		\node at (15:1.2) {\footnotesize $u_1$};
		\node at (75:1.2) {\footnotesize $u_2$};
		\node at (135:1.2) {\footnotesize $u_3$};
		\node at (195:1.2) {\footnotesize $u_4$};
		\node at (255:1.2) {\footnotesize $u_5$};
		\node at (315:1.2) {\footnotesize $u_6$};
		\node at (0:2.35) {\footnotesize $v_1$};
		\node at (60:2.35) {\footnotesize $v_2$};
		\node at (120:2.35) {\footnotesize $v_3$};
		\node at (180:2.35) {\footnotesize $v_4$};
		\node at (240:2.35) {\footnotesize $v_5$};
		\node at (300:2.35) {\footnotesize $v_6$};
		
		\end{tikzpicture}
		\caption{The graph $P(6,2)$.}
		\label{62}
	\end{subfigure}
	\hspace{1cm}
	\begin{subfigure}[t]{0.3\textwidth}
		\centering
		\begin{tikzpicture} [scale = 0.75]
		\pgfmathtruncatemacro{\n}{8}
		\pgfmathtruncatemacro{\m}{\n-1}
		\pgfmathtruncatemacro{\a}{360/\n}

		\foreach \x in {0, ..., \m}
		{
			\pgfmathtruncatemacro{\p}{\x*\a}
			\vertex (x\x) at (\p:1) {};
			\vertex (y\x) at (\p:2) {};
		}
		
		\foreach \x in {0, 1, ..., \m}
		{
			
			\pgfmathtruncatemacro{\y}{mod(\x+1,\n)}
			\pgfmathtruncatemacro{\l}{mod(\x+2,\n)}
			
			\path
			(y\x) edge (y\y)
			(x\x) edge (y\x)
			(x\x) edge (x\l);
		}
		
		\node at (15:1.2) {\footnotesize $u_1$};
		\node at (60:1.2) {\footnotesize $u_2$};
		\node at (105:1.2) {\footnotesize $u_3$};
		\node at (150:1.2) {\footnotesize $u_4$};
		\node at (195:1.2) {\footnotesize $u_5$};
		\node at (240:1.2) {\footnotesize $u_6$};
		\node at (285:1.2) {\footnotesize $u_7$};
		\node at (330:1.2) {\footnotesize $u_8$};
		\node at (0:2.35) {\footnotesize $v_1$};
		\node at (45:2.35) {\footnotesize $v_2$};
		\node at (90:2.35) {\footnotesize $v_3$};
		\node at (135:2.35) {\footnotesize $v_4$};
		\node at (180:2.35) {\footnotesize $v_5$};
		\node at (225:2.35) {\footnotesize $v_6$};
		\node at (270:2.35) {\footnotesize $v_7$};
		\node at (315:2.35) {\footnotesize $v_8$};

		\end{tikzpicture}
		\caption{The graph $P(8,2)$.}
		\label{82}
	\end{subfigure}
	
	\caption{The generalized Petersen graphs $P(6,2)$ and $P(8,2)$.}
	\label{Petersen}
\end{figure}

\begin{theorem}
If $n \geq 5$, then $P(n,2)$ is irreducible.
\end{theorem}

\begin{proof}
Let $n \geq 5$, and let $G=P(n,2)$.  Since it is known that $\beta(P(n,2))=3$~\cite{Javaidetal2008}, it suffices to show that $\tau(G)\geq 3$.  Suppose, to the contrary, that there exists a set of edges $\mathcal{E}\subseteq E\left(\overline{G}\right)$ such that $\beta(G+\mathcal{E})=2$.  Let $\{x,y\}$ be a basis for $G+\mathcal{E}$.  We consider several cases, using the fact that $G$ is $3$-regular throughout.
	
\noindent \textbf{Case 1:} $n \geq 10$, or $n\geq 5$ and $n$ is odd.
	
\noindent In this case, for each vertex $v \in V(G)$, we have $|N_2^G(v)|=6$. If no edge of $\mathcal{E}$ is incident with $x$, then  $N_2^{G+\mathcal{E}}(x) \geq 6$. By Lemma~\ref{2dim}, this is not possible. Hence, there is an edge in $\mathcal{E}$ incident with $x$, and some vertex of $G$ in $N_2^G(x)$. However, then $N_1^{G+\mathcal{E}}(x)\geq 4$.  Again, by Lemma~\ref{2dim}, this is impossible. We conclude that $\tau(G)=3$.

\smallskip
	
\noindent	\textbf{Case 2:} $n=6$.
	
\noindent The graph $G$ is depicted in Figure \ref{62}. For each vertex $v$ on the outer ring, we have $|N_2^G(v)|=6$. Hence, by the same argument as in Case 1, we see that $x$ and $y$ must belong to the inner ring of $G$. Without loss of generality, we may assume that $x=u_1$, and $y=u_i$ for some $i\in\{2,3,4\}$.  In each case, it is straightforward to see that $N_2^G(x)\cap N_2^G(y)$ contains some vertex $v$.  Since $\{x,y\}$ resolves $G+\mathcal{E}$, we see that $\mathcal{E}$ must contain one of the edges $xv$ or $yv$.  But then either $N_1^{G+\mathcal{E}}(x)\geq 4$, or $N_1^{G+\mathcal{E}}(y)\geq 4$, and this is impossible by Lemma~\ref{2dim}.

\smallskip
	
\noindent	\textbf{Case 3:} $n=8$.
	
\noindent The graph $G$ is shown in Figure \ref{82}. As in Case $2$, we see that $x$ and $y$ must belong to the inner ring of $G$.  Without loss of generality, we may assume that $x=u_1$ and $y=u_i$ for some $i\in\{2,3,4,5\}$.  In each case, one can show that $N_2^G(x)\cap N_2^G(y)$ is nonempty, and we reach a contradiction as in Case 2.
\end{proof}

\subsection{Irreducible graphs of given order and dimension} \label{prop}

In this subsection, we prove that an irreducible graph of any given dimension and any order exceeding this dimension exists. We begin with some theorems that will be helpful in constructing irreducible graphs from other irreducible graphs. Graphs of diameter $2$ and the join operation will play an important role in our constructions. Recall that the join of graphs $G$ and $H$ is denoted $G \vee H$.


\begin{lemma} \label{nobar}
	If $\mathrm{diam}(G)\leq2$, then
	\begin{enumerate}
	\item $\beta\left(G \vee \overline{K_2}\right) = \beta(G)+1$; and
	\item $\beta(G \vee K_2)\in \{\beta(G)+1,\beta(G)+2\}$
\end{enumerate}
\end{lemma}

\begin{proof}
	We prove only (i); the proof of (ii) is similar.  Let $H=G \vee \overline{K_2}$. Let $W$ be a basis for $H$.
	Since $\mathrm{diam}(G)=2$, we have $d_H(u,v)= d_G(u,v)$ for all vertices $u,v\in V(G)$. Moreover, neither of the two vertices in $\overline{K_2}$ resolves any pair of vertices that belong to $G$.
	So, $W \cap V(G)$ resolves $G$, and hence $|W \cap V(G)| \ge \beta(G)$. Moreover, $W$ contains at least one of the two vertices joined to every vertex of $G$; otherwise this pair is not resolved by $W$. Thus $|W| \ge \beta(G)+1$.
Now, if $v \in V\left(\overline{K_2}\right)$, and $W_G$ is a basis for $G$, then $W_G \cup \{v\}$ is a resolving set for $H$, since $W_G$ resolves each pair of vertices in $V(G)$, and $v$ resolves any pair of vertices containing at least one vertex from $\overline{K_2}$. So $\beta(H) = \beta(G) +1$.
\end{proof}

If $G$ is an arbitrary graph of diameter at most $2$, note that both $\beta(G\vee K_2)=\beta(G)+1$ and $\beta(G\vee K_2)=\beta(G)+2$ are possible.  For example, it is straightforward to verify that $\beta(K_n\vee K_2)=\beta(K_n)+2$ for all $n\geq 1$, while $\beta(C_4\vee K_2)=\beta(C_4)+1$.  If $G$ has diameter greater than $2$, then it is possible that $\beta(G\vee \overline{K_2})$ and $\beta(G\vee K_2)$ are both larger than $\beta(G)+2$.  For example, let $T$ be the tree obtained from $K_{1,3}$ by subdividing every edge exactly twice.  Then one can verify that $\beta(T)=2$, while $\beta(T\vee \overline{K_2})=\beta(T\vee K_2)=5$.  See Figure~\ref{SubdividedK13Basis} for an illustration of a metric basis of $T\vee\overline{K_2}$.  (The white vertices also form a metric basis of $T\vee K_2$.)

Next, we show that when we join $\overline{K_2}$ to an irreducible graph of diameter $2$, the resulting graph is also irreducible.  This is an important tool in the proof of the main result of this section.

\begin{figure}
	\centering
	\begin{subfigure}[t]{0.45\textwidth}
		\centering
		\begin{tikzpicture} [scale=0.6]
		
		\hollowvertex (0) at (0,0) {};
		
		\foreach \x in {0,1,2}
		{
			\pgfmathtruncatemacro{\p}{\x*120}
			\vertex (1\x) at (30+\p:1) {};
			\hollowvertex (2\x) at (30+\p:2) {};
			\vertex (3\x) at (30+\p:3) {};
			\path 
			(0) edge (1\x)
			(1\x) edge (2\x)
			(2\x) edge (3\x); 
		}
		\hollowvertex (a) at (-30:3) {};
		\vertex (b) at (210:3) {};
		\begin{scope}[shift={(-30:3)}]
		\foreach \x in {90,102,114,126,138,150,162,174,186,198,210}
		{
			\draw (a) -- (\x:1);
		}		
		\end{scope}
		\begin{scope}[shift={(210:3)}]
		\foreach \x in {90,102,114,126,138,150,162,174,186,198,210}
		{
			\draw (b) -- (-120+\x:1);
		}		
		\end{scope}
		
		\end{tikzpicture}
		\caption{The graph $T\vee\overline{K_2}$, with the vertices of a metric basis coloured white.}
		\label{SubdividedK13Basis}
	\end{subfigure}
	\ \ 
	\begin{subfigure}[t]{0.45\textwidth}
		\centering
		\begin{tikzpicture} [scale=0.6]
		
		\vertex (0) at (0,0) {};
		
		\foreach \x in {0,1,2}
		{
			\pgfmathtruncatemacro{\p}{\x*120}
			\vertex (1\x) at (30+\p:1) {};
			\vertex (2\x) at (30+\p:2) {};
			\hollowvertex (3\x) at (30+\p:3) {};
			\path 
			(0) edge (1\x)
			(1\x) edge (2\x)
			(2\x) edge (3\x); 
		}
		\path[dashed]
		(30) edge[bend right =40] (11)
		(30) edge[bend right =40] (10)
		(31) edge[bend right =40] (12)
		(31) edge[bend right =40] (11)
		(32) edge[bend right =40] (10)
		(32) edge[bend right =40] (12);
		
		\hollowvertex (a) at (-30:3) {};
		\vertex (b) at (210:3) {};
		\begin{scope}[shift={(-30:3)}]
		\foreach \x in {90,102,114,126,138,150,162,174,186,198,210}
		{
			\draw (a) -- (\x:1);
		}		
		\end{scope}
		\begin{scope}[shift={(210:3)}]
		\foreach \x in {90,102,114,126,138,150,162,174,186,198,210}
		{
			\draw (b) -- (-120+\x:1);
		}		
		\end{scope}
		
		\end{tikzpicture}
		\caption{A threshold graph for $T\vee\overline{K_2}$, with the vertices of a metric basis coloured white.}
		\label{SubdividedK13Threshold}
	\end{subfigure}
	
	\caption{The graph $T\vee \overline{K_2}$, where $T$ is the tree obtained from $K_{1,3}$ by subdividing every edge exactly twice.  (For clarity, only one end of each edge joining a vertex of $T$ to a vertex of $\overline{K_2}$ is drawn.)  Note that $\beta(T)=\tau(T)=2$, while $\beta(T\vee \overline{K_2})=5$ and $\tau(T\vee\overline{K_2})=4$.}
	\label{SubdividedK13}
\end{figure}
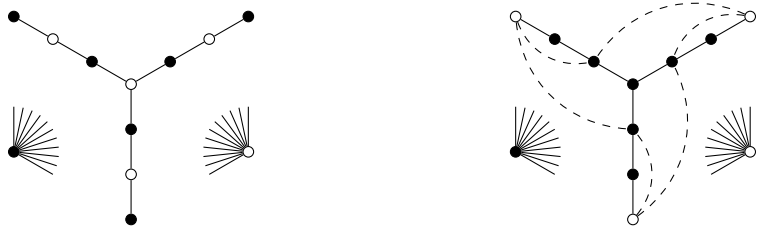

\begin{theorem} \label{k2}
Let $G$ be an irreducible graph, with $\mathrm{diam}(G)\leq 2$. Then $G \vee \overline{K_2}$ is also irreducible, with $\tau\left(G\vee \overline{K_2}\right) = \beta(G)+1$.
\end{theorem}

\begin{proof} 		
Let $H=G\vee\overline{K_2}$. By Lemma~\ref{nobar}, we have that $\beta(H) = \beta(G)+1$.
	
	We now argue that $H$ is irreducible. Suppose towards a contradiction that there is a set $\{e_1, \ldots, e_k\}=\mathcal{E}$ of edges such that $\beta(H + \mathcal{E}) < \beta(H)$. By the left inequality of Lemma~\ref{nobar}, we may assume that all of the edges in $\mathcal{E}$ join vertices of $G$. We claim that $\beta(G+ \mathcal{E}) < \beta(G)$. Let $H' = H + \mathcal{E}$, and $G'=G + \mathcal{E}$. Let $W$ be a basis for $H'$. From the proof of Lemma~\ref{nobar}, we know that exactly one vertex from $V\left(\overline{K_2}\right)$, say $v$, is in $W$. Also, note that $H'=G'\vee \overline{K_2}$. Let $W'=W-\{v\}$. Then $W'$ must resolve $G'$, since for all $u,v \in V(G)$ we have $d_{H'}(u,v) = d_{G'}(u,v)$. So $\beta(G') \leq |W'| = \beta(H') - 1 < \beta(H) - 1 = \beta(G)$. But then $G$ is reducible, a contradiction.
\end{proof}

While there are some irreducible graphs of diameter greater than $2$ for which the conclusion of Theorem~\ref{k2} holds (e.g., every graph obtained from $K_{1,3}$ by subdividing every edge at most once), the conclusion of Theorem~\ref{k2} need not hold in general for graphs of diameter greater than $2$.  For example, the tree $T$ obtained by subdividing every edge of $K_{1,3}$ exactly twice is an irreducible graph of dimension $2$, but one can verify that $\beta(T\vee\overline{K_2})=5$ and $\tau(T\vee\overline{K_2})=4$ (see Figure~\ref{SubdividedK13}).

We next describe two infinite families of irreducible graphs that will be used to establish the main result of this section. For the first of these, we use the function $g$ defined in Section~\ref{preliminaries}. For every $n\geq 2$, let $A_n=K_{g(n)}$ and let $B_n=K_{n-g(n)}$. Apply the Reverse Shortlex Assignment Algorithm to the disjoint union $A_n\cup B_n$ with $W=V(A_n)$ and $P=V(B_n)$.  Let $\mathcal{E}$ be the edges output by the algorithm.  Define $S_n:=(A_n\cup B_n)+\mathcal{E}$.  Note that $S_n$ has order $n$, and that the vertices of $A_n$ form a basis for $S_n$.  Hence $S_n$ has order $n$ and dimension $g(n)$. Since some vertex of $A_n$ is assigned the entire set $V(B_n)$ by the Reverse Shortlex Assignment Algorithm, this vertex is universal in $S_n$. We conclude that $S_n$ has diameter at most $2$.  By Lemma~\ref{weakbound}, we conclude that $S_n$ is irreducible.  Figure~\ref{ring3} shows the graph $S_8$, with the vertices of $A_8$ coloured black.

%

For integers $b>1$ and $s\geq 1$, let $F_{b,s}$ be the graph obtained from the disjoint union $K_{2^b} \cup \overline{K_b} \cup P_s$ by joining a leaf of the path $P_s$ to a single vertex of $\overline{K_b}$.  Apply the Shortlex Assignment Algorithm to $F_{b,s}$ with $W=V\left(\overline{K_b}\right)$ and $P=V(K_{2^b})$, and let the output be $\mathcal{E}$.  Define $S_{b,s}:=F_{b,s}+\mathcal{E}$.  Note that $S_{b,s}$ has order $2^b+b+s$, and that the set $V\left(\overline{K_b}\right)$ resolves $S_{b,s}$.  By Lemma~\ref{kn}, $S_{b,s}$ is irreducible.   The graph $S_{2,3}$ is shown in Figure \ref{po3}, with the vertices of the resolving set $\overline{K_2}$ coloured black.

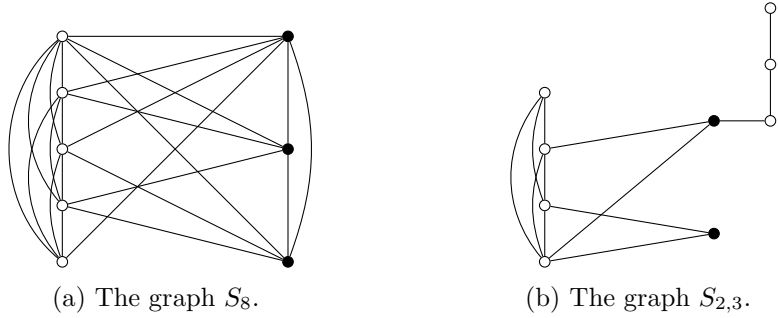
\begin{figure} \label{ringgraph} 
	\centering
	\begin{subfigure}[t]{0.45\textwidth}
		\centering		
		\begin{tikzpicture}[scale=0.75]
		\hollowvertex (x5) at (0,3) {};
		\hollowvertex (x4) at (0,4) {};
		\hollowvertex (x3) at (0,5) {};
		\hollowvertex (x2) at (0,6) {};
		\hollowvertex (x1) at (0,7) {};
		\vertex (v3) at (4,3) {};
		\vertex (v2) at (4,5) {};
		\vertex (v1) at (4,7) {};
		
		\path
		(x1) edge (x2)
		(x2) edge (x3)
		(x3) edge (x4)
		(x4) edge (x5)
		(x1) edge[bend right=20] (x3)
		(x2) edge[bend right=20] (x4)
		(x3) edge[bend right=20] (x5)
		(x1) edge[bend right=40] (x4)
		(x2) edge[bend right=40] (x5)
		(x1) edge[bend right=50] (x5)
		(v1) edge (v2)
		(v2) edge (v3)
		(v1) edge[bend left=20] (v3);

		\path
		(x1) edge (v1)
		(x1) edge (v2)
		(x1) edge (v3)
		(x2) edge (v1)
		(x2) edge (v2)
		(x3) edge (v1)
		(x3) edge (v3)
		(x4) edge (v2)
		(x4) edge (v3)
		(x5) edge (v1)		
		;

		\end{tikzpicture}
		\caption{The graph $S_{8}$.}
		\label{ring3}
	\end{subfigure}
	\begin{subfigure}[t]{0.45\textwidth}
		\centering
		\begin{tikzpicture}[scale = 0.75]
		\hollowvertex (x4) at (0,4) {};
		\hollowvertex (x3) at (0,5) {};
		\hollowvertex (x2) at (0,6) {};
		\hollowvertex (x1) at (0,7) {};
		\vertex (v2) at (3,4.5) {};
		\vertex (v1) at (3,6.5) {};
		\hollowvertex (y1) at (4, 6.5) {};
		\hollowvertex (y2) at (4, 7.5) {};
		\hollowvertex (y3) at (4, 8.5) {};
		
		\path
		(x1) edge (x2)
		(x2) edge (x3)
		(x3) edge (x4)
		(x1) edge[bend right=20] (x3)
		(x2) edge[bend right=20] (x4)
		(x1) edge[bend right=40] (x4);
		
		\path
		(x2) edge (v1)
		(x3) edge (v2)
		(x4) edge (v1)
		(x4) edge (v2)
		(v1) edge (y1)
		(y1) edge (y2)
		(y2) edge (y3)
		;
		\end{tikzpicture}
		\caption{The graph $S_{2,3}$.}
		\label{po3}
	\end{subfigure}
	\caption{The graphs $S_8$ and $S_{2,3}$.  The vertices of a metric basis for each graph are coloured black.}
\end{figure}

%
%
%
%
%
%
%



For graphs $G_1,G_2,\ldots,G_k$, let $\bigvee_{i=1}^k G_i$ be the join of the graphs $G_1,G_2,\ldots, G_k$, i.e., we have $\bigvee_{i=1}^k G_i=G_1\vee G_2\vee\cdots\vee G_k$.

\begin{theorem}
For every integer $n\geq 2$, and every integer $b\in\{1,\dots,n-1\}$, there exists a connected irreducible graph of order $n$ and dimension $b$.
\end{theorem}

\begin{proof}
	If $b=1$, then $P_n$ is a connected irreducible graph of order $n$ and dimension $b$.  From now on, assume that $b > 1$. We consider three different cases.

\smallskip
	
\noindent \textbf{Case 1:} $n>2^b+b$.
	
\noindent Let $s=n-2^b-b$.  Then $S_{b,s}$ is a connected irreducible graph of dimension $b$ and order $n$.
	
\smallskip
	
\noindent \textbf{Case 2:} $2b<n\leq 2^b+b$.

\noindent Let $k$ be the smallest non-negative integer such that 
\[
2^{b-k-1}+b-k \leq n-2k \leq 2^{b-k}+b-k.
\] 
To see that such a $k$ always exists observe first that for $k=b-1$ and $n>2b$ we have $2^{b-k-1} + b-k = 2 < n-2b+2$. So the lower bound holds for some $k$.  Moreover, the upper bound holds for $k=0$. It suffices now to show that there is a $k$ for which both the upper and lower bounds hold. If the upper bound holds for all $k$, then the result follows. Assume now that $(k=)\ell$ is the smallest non-negative integer for which the upper bound fails. Then $n-2\ell > 2^{b-\ell}+b-\ell$. From the above observation, we have $\ell>0$. Thus $n-2(\ell-1) = n-2\ell+2 > 2^{b-\ell}+b-\ell+2=2^{b-(\ell-1)-1}+b-(\ell-1)+1$. Hence for $k=\ell-1$, the lower bound holds. Moreover, the upper bound holds, by assumption, if $k=\ell-1$. So our assertion now follows.
	
	Let $G=S_{n-2k} \vee \left(\bigvee_{i=1}^k\overline{K_2}\right)$. Note first that $G$ has order $n$.  Further, since $2^{b-k-1}+b-k \leq n-2k \leq 2^{b-k}+b-k$ we have $g(n-2k)=b-k$. Hence, we have $\beta(S_{n-2k})=b-k$. By repeated application of Theorem~\ref{k2}, we have $\tau(G)=\beta(G)=\beta(S_{n-2k})+k=b$. Thus, we have shown that $G$ is an irreducible graph with dimension $b$ and order $n$.
	
	\smallskip
	
	\noindent \textbf{Case 3:} $n\leq 2b$.
	
	Let $k=n-1-b$.  Since $b\leq n-1$, we must have $k\geq 0$.  Since $n\leq 2b$, we also have $k\leq n-1-n/2\leq n/2 -1$. Let $G= K_{n-2k}\vee \left(\bigvee_{i=1}^k\overline{K_2}\right)$. Note that $G$ can also be obtained from the complete graph $K_n$ by deleting a matching of size $k$.  By repeated application of Theorem~\ref{k2}, we conclude that $G$ is an irreducible graph of order $n$ and dimension $\beta(K_{n-2k})+k=n-2k-1+k=b$.
\end{proof}

\section{Conclusion}
In this article we gave upper bounds on the threshold dimension of a graph of order $n$, the first in terms of its diameter, and the second in terms of its chromatic number.  The latter bound implies that the threshold dimension of every planar graph of order $n$ is less than $4\log_2n$. We also proved that several infinite families of graphs with constant metric dimension are irreducible. Finally, we showed that there exists an irreducible graph of order $n$ and dimension $b$, for all $n > b \geq 1$.

We posed some questions concerning the computational complexity of the threshold dimension in an earlier paper~\cite{us}.  We add one more such question here.

\begin{question}
	Can the threshold dimension of every cograph be computed in polynomial time?
\end{question}

The problem of characterizing irreducible graphs appears to be quite difficult. Many well-known families of graphs have been characterized in terms of forbidden induced subgraphs (see \cite{BrandstadtLeSpinrad1999}, for example). If $F$ is a graph, then $G$ is $F$-{\em free} if $G$ does not contain $F$ as an induced subgraph. If $\mathcal{F}$ is a family of graphs, then a graph $G$ is \emph{$\mathcal{F}$-free} if $G$ is $F$-free for all $F \in \mathcal{F}$. For example, cographs are exactly the $P_4$-free graphs, and perfect graphs are precisely the $\left\{C_{2k+1}, \overline{C_{2k+1}}\colon\ k\geq 2\right\}$-free graphs. We now observe that the irreducible graphs do not have a forbidden subgraph characterization.

\begin{theorem}
	Let $G$ be a graph. Then there exists an irreducible graph $H$ with $G$ as an induced subgraph.
\end{theorem}

\begin{proof}
Let $G$ be a graph of order $n$. Let $p$ be the smallest positive integer such that $p+n=2^k$ for some positive integer $k$. Let $G'=G \vee K_p$, and let $U=K_k$. Apply the Shortlex Assignment Algorithm to $G'\cup U$ with $W=V(U)$ and $P=V(G')$, ordering the vertices of $G'$ so that a universal vertex of $G'$ appears last. Let $\mathcal{E}$ be the output of the algorithm, and let $H=(G'\cup U)+\mathcal{E}$. Since the last vertex of $G'$ becomes universal in $H$, we have $\mbox{diam}(H)=2$.  By Lemma~\ref{weakbound}, the vertices of $U$ form a basis for $H$, and $H$ is irreducible.
\end{proof}

We conclude with the following question.

\begin{question}
	Is the problem of determining whether a graph is irreducible NP-hard?
\end{question}


\begin{thebibliography}{1}
	
	\bibitem{Belmonteetal2015} R. ~Belmonte, F. V. Fomin, P. A. Golovach, and M. S. Ramanujan, \emph{Metric dimension of bounded width graphs}, In: G.Italiano, G. Pighizzini, and D. Sannella (eds.), Mathematical Foundations of Computer Science 2015 (MFCS 2015). Lecture Notes in Computer Science, 9235, Springer, Berlin, Heidelberg, 115--126.

\bibitem{BrandstadtLeSpinrad1999} A. ~Brandst\"{a}dt, V.-B. ~Le, and J.~P.~ Spinrad, {\em Graph Classes: A Survey}, SIAM monographs on Discrete Mathematics and Applications, 1999.
	
	\bibitem{Caceresetal2007} J. ~C\'aceres, C. ~Hernando, M. ~Mora, I. ~M. ~Pelayo, M. ~L. ~Puertas, C. ~Seara, and D. ~R. ~Wood. \emph{On the metric dimension of Cartesian products of graphs}, SIAM J. Discrete Math. \textbf{21 (2)} (2007) 423--441.
	
	\bibitem{Chartrandetal2000} G. ~Chartrand, L. ~Eroh, M. ~A. Johnson, and O. ~R. Oellermann, \emph{Resolvability in graphs and the metric dimension of a graph}, Discrete Appl. Math. \textbf{105} (2000) 99--113.
	
	\bibitem{HararyMelter1976}  F. ~Harary and R. ~Melter, \emph{The metric dimension of a graph}, Ars Combin. \textbf{2} (1976) 191--195.
	
	\bibitem{Khulleretal1996} S. ~Khuller, B. ~Raghavachari, and A. ~Rosenfeld. \emph{Landmarks in graphs}, Discrete Appl. Math. \textbf{70} (1996) 200--207.
	

	\bibitem{Hernandoetal2005} C. ~Hernando, M. ~Mora, I. ~M. ~ Pelayo, C. ~Seara, J. ~C\'aceres, and M. ~L. ~Puertas. \emph{On the metric dimension of some families of graphs}, Electron. Notes Discrete Math. \textbf{20} (2005) 129--133.
	
	\bibitem{Hernandoetal2010} C. ~Hernando, M. ~Mora, I. ~M. ~Pelayo, C. ~Seara, and D. ~R. ~Wood, \emph{Extremal graph theory for metric dimension and diameter}, Electron. J. Combin. \textbf{17} (2010) \#R30.
	
	\bibitem{Javaidetal2008} I. ~Javaid, M. ~T. ~Rahim, and K. ~Ali, \emph{Families of regular graphs with constant metric dimension}, Util. Math. \textbf{65} (2008) 21--33.
	
	\bibitem{us} L. ~Mol, M. ~J. ~H. ~Murphy, and O. ~R. ~Oellermann, \emph{The threshold dimension of a graph}, preprint (2020).  Available at \url{https://arxiv.org/abs/2001.09168}.
	
	\bibitem{Sudhakaraetal2009} G. ~Sudhakara and A. ~R. ~Hemanth Kumar, \emph{Graphs with metric dimension two -- a characterization}, World Academy of Science, Engineering and Technology \textbf{36} (2009) 622--627.
	
	\bibitem{Slater1975} P. ~J. ~Slater, \emph{Leaves of trees}, Congr. Numer. \textbf{14} (1975) 549--559.
	
	
\end{thebibliography}
\end{document}